\documentclass[12pt]{article}

\usepackage{amsmath,amsfonts,amssymb,amsthm,graphicx,cite}
\usepackage{mathrsfs}
\usepackage{hyperref}

\numberwithin{equation}{section}

\begin{document}

\newtheorem{theorem}{Theorem}[section]
\newtheorem{corollary}[theorem]{Corollary}
\newtheorem{lemma}[theorem]{Lemma}
\newtheorem{proposition}[theorem]{Proposition}

\newcommand{\be}{\begin{equation}}
\newcommand{\ee}{\end{equation}}
\newcommand{\De}{\Delta}
\newcommand{\de}{\delta}
\newcommand{\eps}{\epsilon}
\newcommand{\Z}{{\mathbb Z}}
\newcommand{\N}{{\mathbb N}}
\newcommand{\C}{{\mathbb C}}
\newcommand{\Cs}{{\mathbb C}^{*}}
\newcommand{\R}{{\mathbb R}}
\newcommand{\Q}{{\mathbb Q}}
\newcommand{\T}{{\mathbb T}}
\newcommand{\re}{{\rm Re}\, }
\newcommand{\im}{{\rm Im}\, }
\newcommand{\cC}{{\cal C}}
\newcommand{\cS}{{\cal S}}

\title{Baker--Akhiezer specialisation of joint eigenfunctions for hyperbolic relativistic Calogero--Moser Hamiltonians}

\author{Martin Halln\"as\footnote{E-mail: hallnas@chalmers.se} \\ Department of Mathematical Sciences \\ Chalmers University of Technology and the University of Gothenburg\\ SE-412 96 Gothenburg, Sweden}

\date{\today}

\maketitle

\begin{abstract}
In earlier joint work with Ruijsenaars, we constructed and studied symmetric joint eigenfunctions $J_N$ for quantum Hamiltonians of the hyperbolic relativistic $N$-particle Calogero--Moser system. For generic coupling values, they are non-elementary functions that in the $N=2$ case essentially amount to a `relativistic' generalisation of the conical function specialisation of the Gauss hypergeometric function ${}_2F_1$. In this paper, we consider a discrete set of coupling values for which the solution to the joint eigenvalue problem is known to be given by functions $\psi_N$ of Baker--Akhiezer type, which are elementary, but highly nontrivial, functions. Specifically, we show that $J_N$ essentially amounts to the antisymmetrisation of $\psi_N$ and, as a byproduct, we obtain a recursive construction of $\psi_N$ in terms of an iterated residue formula.
\end{abstract}

\section{Introduction}
Relativistic generalizations of $N$-particle Calogero--Moser systems were originally conceived by Ruijsenaars to provide integrable quantum mechanical descriptions of relativistic quantum field theories in 1+1 spacetime dimensions, such as the quantum sine-Gordon theory, restricted to a $N$-particle sector \cite{RS86,Rui01}.
For this purpose, the relativistic system of hyperbolic type, given by the formally self-adjoint and pairwise commutating analytic difference operators (A$\De$Os)
\be
\label{Sr}
S_r(g;x)\equiv \sum_{\substack{I\subset\lbrace 1,\ldots,N\rbrace\\ |I|=r}}\prod_{\substack{j\in I\\ k\notin I}}f_-(g;x_j-x_k)\prod_{l\in I}\exp(-i\hbar \beta\partial_{x_l})\prod_{\substack{j\in I\\ k\notin I}}f_+(g;x_j-x_k),
\ee
where
\be
f_{\pm}(g;z)= \left(\frac{\sinh(\mu(z\pm i\beta g)/2)}{\sinh(\mu z/2)}\right)^{1/2},
\ee
is of particular importance.
Here, $r=1,\ldots,N$, it is natural to view $\beta>0$ as $1/mc$, with $m$ the particle mass and $c$ the speed of light, and in the nonrelativistic limit $c\to\infty$ pairwise commuting Hamiltonians of the nonrelativistic hyperbolic Calogero--Moser system are recovered; see~e.g.~the surveys \cite{Rui99,Hal25} and references therein.

In a series of joint papers with Ruijsenaars \cite{HR14,HR18a,HR21}, we developed a recursive scheme producing explicit symmetric joint eigenfunctions of the A$\De$Os \eqref{Sr}. In addition to various analyticity and invariance properties, we deduced an explicit formula for dominant asymptotics deep in a Weyl chamber and thereby proved that particles in the hyperbolic relativistic Calogero--Moser system exhibit soliton scattering (i.e.~conservation of momenta and factorization of the scattering matrix).

Further fundamental properties of the joint eigenfunctions were obtained by Belousov et al.~in a number of recent papers \cite{BDKK24a,BDKK24b,BDKK24c,BDKK24d}. Their results include integral equations, a reflection symmetry of the coupling constant, self-duality under interchange of geometric and spectral variables as well as orthogonality and completeness relations.

On a formal level, Ruijsenaars' hyperbolic A$\De$Os \eqref{Sr} are closely related to Macdonald's q-difference operators
\be
\label{Dr}
D^r_N(z;q,t)\equiv t^{r(r-1)/2}\sum_{\substack{I\subset\lbrace 1,\ldots,N\rbrace\\ |I|=r}}\prod_{\substack{j\in I\\ k\notin I}}\frac{tx_j-x_k}{x_j-x_k}\prod_{l\in I}T_{q,x_l},
\ee
where $(T_{q,x_l}f)(x_1,\ldots,x_k,\ldots,x_N)=f(x_1,\ldots,qx_l,\ldots,x_N)$. (The precise relationship can be gleaned from \eqref{Ars} and \eqref{DrN}.) They act on the space of symmetric polynomials in $N$ variables, on which they are simultaneously diagonalised by the symmetric $GL_N$ type Macdonald polynomials; see e.g.~\cite{Mac95,Nou23,Sto21}. For parameter values of the form $t=q^m$, $m\in\Z$, Etingof and Styrkas \cite{ES98} and Chalykh \cite{Cha02} constructed and studied non-symmetric joint eigenfunctions of Baker--Akhiezer (BA) type and obtained, in particular, a generalized Weyl character formula for the Macdonald polynomials, first conjectured by Felder and Varchenko \cite{FV97}, where the BA-function replaces the exponential function; see also Chalykh and Etingof \cite{CE13} for further related results and references.

In this paper, we restrict attention to the discrete set of coupling values $g=m\hbar$, $m\in\mathbb{Z}_+$, and show in Thm.~\ref{Thm:PhiN} that the joint eigenfunctions from \cite{HR14} of Ruijsenaars' A$\De$Os $S_r(m\hbar;x)$ are obtained by antisymmetrization of the (self-dual) BA-function associated with the Macdonald operators $D^r_N(z;q^2,q^{-2m})$. As a by-product of our proof, we obtain in Prop.~\ref{Prop:psiN} a recursive construction of the BA-function by an iterated residue formula.

For the coupling values $g=m\hbar$ under consideration, our main result provides an expansion of symmetric joint $S_r$-eigenfunctions in terms of eigenfunctions that are `asymptotically free' deep inside Weyl chambers $x_{\sigma(1)}>x_{\sigma(2)}>\cdots>x_{\sigma(N)}$, $\sigma\in S_N$. In the case of the Macdonald operators and generic parameter values, such asymptotically free solutions, often referred to as (q-)Harisch--Chandra series, have been studied in detail by Letzter and Stokman \cite{LS08}, van Meer and Stokman \cite{vMS10} as well as Noumi and Shiraishi \cite{NS12}; and Stokman \cite{Sto14} derived a corresponding ($c$-function) expansion of Cherednik's basic hypergeometric function associated to root systems. We should also note that Bullimore et al.~\cite{BKK15} showed that the partition function of the superconformal field theory $T[U(N)]$, as introduced by Gaiotto and Witten \cite{GW09}, displays the very same recursive structure as the joint $S_r$-eigenfunctions from \cite{HR14}; and, identifying integration with summing residues at all (infinitely many) poles above the real axis, they obtained an expression in terms of series of an asympototically free type. Furthermore, in the $N=2$ case, di Francesco et al.~\cite{FKKSS24} recently obtained a number of interesting results on the $S_r$-eigenfunctions, including a similar residue computation where the relevant series are identified as (q-)Harish--Chandra series. To establish the precise connection between these expansion results and the recursive construction in \cite{HR14} is, I believe, an interesting problem to which I plan to return to elsewhere. In the $BC_1$ case, the analogous connection between integral representations and series expansions was established by van de Bult et al.~\cite{vdBRS07}.

The remainder of the paper is structured as follows: In Section 2, we briefly recall definitions and key properties pertaining to joint eigenfunctions of the Ruijsenaars operators \eqref{Sr} and BA-(eigen)functions for the Macdonald operators \eqref{Dr}; and, in Section 3, we give the precise formulations and proofs of our results.

\section{Preliminaries}

\subsection{Joint eigenfunctions} 
To begin with, we reparametrize the two length scales of the A$\De$Os \eqref{Sr} as
\be
a_+\equiv 2\pi/\mu > 0,\quad \mathrm{(imaginary~period/interaction~length)}
\ee
and
\be
a_-\equiv \hbar\beta > 0,\quad \mathrm{(shift~step~size/Compton~wavelength)}
\ee
and replace the coupling $g$ by the parameter
\be
b\equiv \beta g.
\ee
Rewriting \eqref{Sr} in terms of the parameters $a_\pm$, the coefficients become manifestly $ia_+$-periodic. It follows that the A$\De$Os obtained after the interchange $a_+\leftrightarrow a_-$ commute with the given ones. In this way, we obtain $2N$ pairwise commuting A$\De$Os $H_{r,\de}(b;x)$, with $r=1,\ldots,N$ and $\delta=+,-$, and where $H_{r,+}(b;x)=S_r(g;x)$ with $b$ and $g$ related as above; cf.~Eq.~(1.7) in \cite{HR14}.

It is often convenient to work with similarity transforms of the A$\De$Os $H_{r,\de}$ by the weight function
\be\label{W}
W_N(b;x)\equiv \prod_{1\leq j<k\leq N}\prod_{\de=+,-}\frac{G(\de(x_j-x_k)+i(a_++a_-)/2)}{G(\de(x_j-x_k)+i(a_++a_-)/2-ib)},
\ee
where $G(z)\equiv G(a_+,a_-;z)$ is the hyperbolic Gamma function from \cite{Rui97}, which is $a_+\leftrightarrow a_-$ invariant, meromorphic in $z$ and satisfies the analytic difference equations
\be\label{Gades}
\frac{G(z+ia_\de/2)}{G(z-ia_\de/2)} = 2\cosh(\pi z/a_{-\de}),\ \ \ \de = +,-;
\ee
see loc.~cit.~Prop.~III.1 and Prop.~III.2. (Unless needed, we suppress dependence of $G$ and functions constructed from $G$ on the parameters $a_\pm$.) Indeed, using these difference equations, it is readily verified that
\be\label{Ars}
\begin{split}
A_{r,\delta}(b;x) &\equiv W_N(b;x)^{-1/2}H_{r,\de}(b;x)W_N(b;x)^{1/2}\\
&= \sum_{\substack{I\subset\lbrace 1,\ldots,N\rbrace\\ |I|=r}}\prod_{\substack{j\in I\\ k\notin I}}\frac{\sinh(\pi(x_j-x_k-ib)/a_\de)}{\sinh(\pi(x_j-x_k)/a_\de)}\prod_{l\in I}\exp(- ia_{-\delta}\partial_{x_l}),
\end{split}
\ee
which, in contrast to $H_{r,\de}$, act on the space of meromorphic functions. (We recall that, for $(a_+,a_-,b)\in(0,\infty)^3$ such that $b<a_++a_-$ and $x\in\R^N$, the weight function is regular and positive.)

In \cite{HR14} functions $J_N(b;x,y)$ having the joint eigenfunction property
\be
\label{JNEq}
A_{r,\de}(b;x)J_N(b;x,y) = J_N(b;x,y)\sum_{1\leq j_1<\cdots <j_r\leq N}e^{2\pi (y_{j_1}+\cdots +y_{j_r})/a_\de},
\ee
are constructed recursively by an explicit integral formula, with integrand built from $J_{N-1}$, the weight function $W_{N-1}$ as well as the kernel function
\be\label{cSf}
\cS_N^{\sharp}(b;x,y) \equiv \prod_{j=1}^N\prod_{k=1}^{N-1}\frac{G(x_j-y_k-ib/2)}{G(x_j-y_k+ib/2)}.
\ee
More precisely, with $J_1(x,y)\equiv\exp(2\pi i xy/a_+a_-)$ as the starting point for the recursion, $J_N$, $N>1$, is given by
\begin{multline}\label{JN}
J_N(b;x,y) = \frac{\exp\left(\frac{2\pi i}{a_+a_-}y_N\sum_{j=1}^Nx_j\right)}{(N-1)!}\\
\cdot \int_{\R^{N-1}}dzW_{N-1}(b;z)\cS^\sharp_N(b;x,z)J_{N-1}(b;z,(y_1-y_N,\ldots,y_{N-1}-y_N)).
\end{multline}

From Prop.~III.2 in \cite{Rui97}, we recall that the hyperbolic gamma function is scale invariant, in the sense that
\be\label{sc}
G(\lambda a_+,\lambda a_-;\lambda z) = G(a_+,a_-;z),\quad \lambda\in(0,\infty).
\ee
Since $J_N$ is constructed almost entirely from $G(z)$, it is easily seen to have the invariance properties
\be
J_N(a_+,a_-,b;x,y) = J_N(a_-,a_+,b;x,y),
\ee
\be
J_N(\lambda a_+,\lambda a_-,\lambda b;\lambda x,\lambda y) = \lambda^{N(N-1)/2}J_N(a_+,a_-,b;x,y).
\ee
Hence, we may and shall restrict attention to 
\be
\label{apmSpec}
a_+ = 1,\ \ \ a_-\equiv a\geq 1,
\ee
without any loss of generality. Moreover, to facilitate the comparison with the pertinent BA-function, it is expedient to renormalize $J_N$ and introduce the function
\be\label{PhiN}
\Phi_N(a,b;x,y)\equiv \left(\frac{G(1,a;ib-i(1+a)/2)}{\sqrt{a}}\right)^{N-1}J_N(1,a,b;x,y),
\ee
which, in particular, satisfies the simple duality relation
\be
\label{dualRel}
\Phi_N(a,b;x,y) = \Phi_N(a,1+a-b;y,x),
\ee
as conjectured in \cite{HR14}, verified for $N\leq 3$ in \cite{HR18a} and proven for arbitrary $N$ in \cite{BDKK24b} (see Thm.~5 and Eqs.~(1.53)--(1.54)).

\subsection{Baker--Akhiezer functions}
Take $q\in\mathbb{C}^\times\equiv\C\setminus\{0\}$, $m\in\mathbb{Z}_+$ and let $\psi_N$ be a function of $x,y\in\mathbb{C}^N$ that is of the form
\be
\label{psiN}
\psi_N(x,y) = q^{2(x,y)}\sum_\nu \psi_{N,\nu}(x)q^{2(\nu,y)},
\ee
where the sum extends over weight vectors
\be\label{nu}
\nu = \sum_{1\leq j<k\leq N}\left(\frac{m}{2}-l_{jk}\right)(e_j-e_k),\ \ \ l_{jk} = 0,\ldots,m,
\ee
with $\{e_1,\ldots,e_N\}$ the standard ON-basis in $\mathbb{R}^N$. Suppose, in addition, that $\psi_N$ satisfies the vanishing condition
\be
\label{psiNqinv}
\psi_N(x,y+s(e_j-e_k)/2) - \psi_N(x,y-s(e_j-e_k)/2) = 0\ \ \text{when}\ \ q^{2(y_j-y_k)} = 1,
\ee
for all $1\leq j<k\leq N$ and $s=1,\ldots,m$. Then it is called a Baker--Akhiezer (BA) function associated with the root system $A_{N-1}$ and (positive) integer parameter $m$. Existence is established by different constructions in Sect.~5 of \cite{ES98} and Sect.~3.2 of \cite{Cha02}. We note that the iterated residue formula obtained in this paper yields an additional (constructive) existence proof.

Assuming $q$ is not a root of unity, we know from \cite{ES98,Cha02} that a function $\psi_N$ with the above properties is unique up to a choice of normalisation; and that, imposing the normalisation condition
\be
\label{psiNrhom}
\psi_{N,\rho_N}(x) = \prod_{1\leq j<k\leq N}\prod_{n=1}^m[n+x_k-x_j],\ \ \ [z]\equiv q^z-q^{-z},
\ee
yields a self-dual BA function, in the sense that $\psi_N(x,y)=\psi_N(y,x)$. Here, we have used the standard notation
\be
\label{rhoNm}
\rho_N\equiv \rho_N(m) = \frac{m}{2}\sum_{1\leq j<k\leq N}(e_j-e_k) = \frac{m}{2}\sum_{j=1}^N(N-2j+1)e_j,
\ee
obtained by setting all $l_{ij}=0$ in \eqref{nu}. The self-dual BA-function satisfies the bispectral system of $q$-difference equations
\begin{align}
\label{MacDEqs}
D^r_N(q^{2y};q^2,q^{-2m})\psi(x,y) &= \psi(x,y)\sum_{1\leq j_1<\cdots <j_r\leq N}q^{2 (x_{j_1}+\cdots +x_{j_r})},\\
D^r_N(q^{2x};q^2,q^{-2m})\psi(x,y) &= \psi(x,y)\sum_{1\leq j_1<\cdots <j_r\leq N}q^{2 (y_{j_1}+\cdots +y_{j_r})},
\end{align}
with $r=1,\ldots,N$, $q^{2x}\equiv (q^{2x_1},\ldots,q^{2x_N})$ and $q^{2y}\equiv (q^{2y_1},\ldots,q^{2y_N})$.

\section{Baker--Akhiezer specialisation}
When $b=m\in\mathbb{Z}_+$, the weight function $W_N$ \eqref{W} and kernel function $\cS^\sharp_N$ \eqref{cSf} reduce to elementary functions. More precisely, from the difference equation \eqref{Gades} satisfied by $G(z)$, we see that
\be
\label{cW}
W_N(1,a,m;x) = \prod_{1\leq j<k\leq N}\sinh\left(\frac{\pi}{a}(x_k-x_j)\right)\prod_{n=-m+1}^{m-1} 2 \sinh\left(\frac{\pi}{a}(x_k-x_j+in)\right)
\ee
and
\begin{equation}
\label{cK}
\cS^\sharp_N(1,a,m;x,y) = \prod_{j=1}^N\prod_{k=1}^{N-1}\prod_{n=0}^{m-1}\frac{1}{2\cosh\big(\frac{\pi}{a}(x_j-y_k+i(m-2n-1)/2)\big)}.
\end{equation}
Up to an overall numerical factor, $W_N(1,a,m;x)$ equals the specialisation $\De(q^{2ix};q^2,q^{2m})$ of Macdonald's weight function $\De$, while $\cS^\sharp_N(1,a,m;x,y)$ is obtained by a similar specialisation of his product function $\Pi$ after removing a factor depending only on $\sum_{j=1}^N(x_j-y_j)$; cf.~Eqs.~(9.2) and (2.5), respectively, in Section VI of \cite{Mac95}. Furthermore, \eqref{Gades} and the formula $G(1,a;i(1-a)/2)=1/\sqrt{a}$ (see Eq.~(3.38) in \cite{Rui97}) entail that
\be
\frac{G(1,a;im-i(1+a)/2)}{\sqrt{a}} = \frac{\prod_{n=1}^{m-1}2\sin(\pi n/a)}{a}.
\ee
Consequently, the $b=m$ specialisation of $\Phi_N$ is given by recursively by
\begin{multline}
\label{PhiNRec}
\Phi_N(m;x,y) = \frac{\prod_{n=1}^{m-1}2\sin(\pi n/a)}{a(N-1)!}
e^{\frac{2\pi i}{a}y_N(x_1+\cdots+x_N)}\\
\cdot \int_{\mathbb{R}^{N-1}}dz \prod_{1\leq j<k\leq N-1}\sinh\left(\frac{\pi}{a}(z_k-z_j)\right)\prod_{n=-m+1}^{m-1} 2 \sinh\left(\frac{\pi}{a}(z_k-z_j+in)\right)\\
\cdot \frac{\Phi_{N-1}(m;z,(y_1-y_N,\ldots,y_{N-1}-y_N))}{\prod_{j=1}^N\prod_{k=1}^{N-1}\prod_{n=0}^{m-1}2\cosh\big(\frac{\pi}{a}(x_j-z_k+i(m-2n-1)/2)\big)}
\end{multline}
with $\Phi_1(x,y)=\exp\left(\frac{2\pi}{a}xy\right)$. We note that this is effectively a basic hypergeometric integral, since the weight and kernel functions can be expressed in terms of the product
\be
\prod_{n=0}^{m-1}(q^{n+z}-q^{-n-z}) = (-1)^m q^{-m(m-1)/2-mz}\frac{(q^{2z};q^2)_\infty}{(q^{2(m+z)};q^2)_\infty},
\ee
with
\be
\label{q}
q = e^{-i\pi/a},\ \ \ a\in\mathbb{R}.
\ee
To be precise, while the (modified) $q$-Gamma function $(q^{2z};q^2)_\infty$ is only well-defined for $|q|<1$, the above ratio clearly extends to $|q|=1$. For accounts of (multivariable) basic hypergeometric integrals, see e.g.~the overview by Koornwinder and Stokman \cite{KS21} and the survey by Schlosser \cite{Sch21}.

Upon setting $(a_+,a_-,b)=(1,a,m)$, the $a$-antiperiodicity of the functions involved entails that the eigenvalue equations \eqref{JNEq}, which are are also satisfied by $\Phi_N$, become trivial for $\de=+$. As we now show, the remaining equations, corresponding to $\de=-$, are closely related to the Macdonald eigenvalue equations \eqref{MacDEqs}, of which the BA-functions are solutions.

Letting
\be
\label{deN}
\mathcal{\de}_N(x;m) = \prod_{1\leq j<k\leq N}\prod_{n=-m}^{m} [x_j-x_k+n],
\ee
we recall the well-known similarity transform
\be
\mathcal{\de}_N(x;m-1)^{-1}D^r_N(q^{2x};q^2,q^{-2(m-1)})\mathcal{\de}_N(x;m-1) = D^r_N(q^{2x};q^2,q^{2m}),
\ee
which is easily verified by a direct computation; and with the parametrisation \eqref{q} in effect it is readily seen that
\be
\label{DrN}
q^{mr(1-N)}D^r_N(q^{2ix};q^2,q^{2m}) = \sum_{\substack{I\subset\lbrace 1,\ldots,N\rbrace\\ |I|=r}}\prod_{\substack{j\in I\\ k\notin I}}\frac{\sinh\big(\frac{\pi}{a}(x_j-x_k-im)\big)}{\sinh\big(\frac{\pi}{a}(x_j-x_k)\big)}\prod_{l\in I}\exp(-i\partial_{x_l}).
\ee
Note that the right-hand side coincides with the A$\De$O $A_{r,-}(1,a,m;x)$ \eqref{Ars}.

In what follows, it will be important to keep track of dependence on the parameters $a$ or $q$ as well as $m$. Therefore, we write $\psi_N(q,m;x,y)$ for the self-dual BA-function characterised by \eqref{psiN}--\eqref{psiNrhom}. By the joint eigenfunction properties reviewed in the previous section, it is natural to expect that $\Phi_N(a,m;x,y)$ is proportional to the symmetrisation of $\mathcal{\de}_N(ix;m-1)^{-1}\psi_N(e^{-i\pi/a},m-1;ix,iy)$. Indeed, since both functions satisfy the same set of joint eigenvalue equations with respect to the $S_N$-invariant A$\De$Os \eqref{DrN}, it is plausible that $\Phi_N(a,m;x,y)$ can be expressed as a linear combination of the functions $\mathcal{\de}_N(ix;m-1)^{-1}\psi_N(e^{-i\pi/a},m-1;i\sigma(x),iy)$, $\sigma\in S_N$, where we have used the fact that
\be
\mathcal{\de}_N(i\sigma(x);m-1) = (-)^\sigma \mathcal{\de}_N(ix;m-1),\ \ \ \sigma\in S_N,
\ee
which is clear from \eqref{deN}. Moreover, given that $\Phi_N(a,m;x,y)$ is manifestly symmetric in $x$ (cf.~\eqref{JN} and \eqref{PhiN}), the coefficients in the linear combination should be of the form $c(y)(-)^\sigma$, with $c$ independent of $\sigma$. In this section, we substantiate this expectation by proving the following theorem.

\begin{theorem}
\label{Thm:PhiN}
Let $m\in\mathbb{Z}_+$ and let $a>m-1$ be such that $\exp(-i\pi/a)$ is not a root of unity. Then, we have
\begin{multline}
\Phi_N(a,m;x,y)\\
= C_N(a,m)\frac{\sum_{\sigma\in S_N}(-)^\sigma \psi_N(e^{-i\pi/a},m-1;i\sigma(x),iy)}{\prod_{1\leq j<k\leq N}2\sinh\big(\frac{\pi}{a}(y_k-y_j)\big)\prod_{n=-m+1}^{m-1}2\sinh\big(\frac{\pi}{a}(x_k-x_j+in)\big)}
\end{multline}
with the constant
\be
C_N(a,m) = (i^m)^{N(N-1)/2}\left(\frac{a}{\prod_{n=1}^{m-1}2\sin(\pi n/a)}\right)^{(N-1)(N-2)/2}.
\ee
\end{theorem}

The first main observation in our proof of Thm.~\ref{Thm:PhiN}, as detailed below, is that the product function $\Phi_N(a,m;x,y) \prod_{1\leq j<k\leq N}2\sinh(\pi(y_j-y_k))$ can be rewritten as the symmetrisation of a function $\varphi_N(a,m;x,y)$ that is given recursively by an iterated residue formula.

More specifically, the formula involves denominator factors $\sinh\big(\frac{\pi}{a}(x_j-z_k+i(m-2n-1)/2)\big)$, producing poles at
\be
z_k = x_j+i(m-2n-1)/2+i\ell a,\ \ \ n = 0,\ldots,m-1,\ \ \ell\in\mathbb{Z},
\ee
where $j=1,\ldots,N$ and $k=1,\ldots,N-1$. Assuming that $\re(x_j-x_k)\neq 0$ for all $1\leq j<k\leq N$ and $a>m-1$, we can find contours $\gamma_k$ that only encircle the poles corresponding to a fixed value of $k=1,\ldots,N-1$ and  $\ell=0$; see the discussions preceding Props.~\ref{Prop:Phi2Exp} and \ref{Prop:PhiNExp} for explicit examples. Up to normalisation, the recursive formula defining $\varphi_N$ can then be obtained from \eqref{PhiNRec} by the substitutions $\Phi\to\varphi$, $\R^{N-1}\to\gamma_1\times\cdots\times\gamma_{N-1}$ and $\cosh\to\sinh$.

The second main observation, which more or less completes the proof of Thm.~\ref{Thm:PhiN}, is that $\varphi_N$ essentially amounts to a renormalisation of $\psi_N$ and, as a byproduct, we thus find the following recursive construction of the self-dual BA-function $\psi_N$.

\begin{proposition}
\label{Prop:psiN}
Under the above assumptions on $m$ and $a$, the BA-function $\psi_N$ is obtained from $\psi_{N-1}$ by the iterated residue formula
\begin{multline}
\psi_N(e^{-i\pi/a},m;ix,iy)\\
= \left(\frac{\prod_{n=1}^m\sin(\pi n/a)}{a(-i)^{m+1}}\right)^{N-1} \frac{W_N(1,a,m+1;z)}{\prod_{1\leq j<k\leq N}2\sinh\left(\frac{\pi}{a}(x_k-x_j)\right)} e^{\frac{2\pi i}{a} y_N\sum_{j=1}^Nx_j}\\
\cdot \int_{\underline{\gamma}}dz \frac{\psi_{N-1}(e^{-i\pi/a},m;iz,i(y_1-y_N,\ldots,y_{N-1}-y_N))\prod_{1\leq j<k\leq N-1}2\sinh\left(\frac{\pi}{a}(z_k-z_j)\right)}{\prod_{j=1}^N\prod_{k=1}^{N-1}\prod_{n=0}^{m}2\sinh\big(\frac{\pi}{a}\big(x_j-z_k+\frac{i}{2}(m-2n)\big)\big)},
\end{multline}
where $\underline{\gamma}\equiv \gamma_1\times\cdots\times\gamma_{N-1}$.
\end{proposition}

This result provides a natural generalisation of the iterated residue formulae for BA-functions for rational and trigonometric Calogero--Moser--Sutherland operators obtained by Felder and Veselov \cite{FV09}.

In the two subsections below, we provide our proofs of Thm.~\ref{Thm:PhiN} and Prop.~\ref{Prop:psiN}. To begin with, we establish the $N=2$ cases of the results. These we then use as the basis for proofs by induction on $N$. To shorten some of the formulae involved, we typically suppress dependence on the parameters $a$ and $m$.

\subsection{The $N=2$ case}
Setting $N=2$ in \eqref{PhiNRec}, we deduce
\begin{multline}
\label{Phi2Expr}
\frac{2a\sinh\big(\frac{\pi}{a}(y_1-y_2)\big)}{\prod_{n=1}^{m-1}2\sin(\pi n/a)}\Phi_2(x,y)\\
= \int_{\mathbb{R}}dz\frac{e^{\frac{2\pi i}{a}(y_1-y_2)(z-ia/2)}-e^{\frac{2\pi i}{a}(y_1-y_2)(z+ia/2)}}{\prod_{j=1}^2\prod_{n=0}^{m-1}2\cosh\big(\frac{\pi}{a}(x_j-z+i(m-2n-1)/2)\big)}\\
= (-1)^m\left(\int_{\R-ia/2}dz-\int_{\R+ia/2}dz\right)\frac{e^{\frac{2\pi i}{a}(y_1-y_2)z}}{\prod_{j=1}^2\prod_{n=0}^{m-1}2\sinh\big(\frac{\pi}{a}(x_j-z+i(m-2n-1)/2)\big)}.
\end{multline}
We note that $a>m-1$ and $x,y\in\mathbb{R}^2$ ensures that $\Phi_2$ is well-defined and that the only poles of the integrand located within the strip $|\im z|<a/2$ are
\be
\label{N2Pls}
z = x_j+i(m-2n-1)/2,\ \ \ j = 1,2,\ \ n = 0,\ldots,m-1.
\ee
At first, we assume $x=(x_1,x_2)\in\R^2$ is such that $x_1\neq x_2$, which ensures that we can introduce a contour $\gamma_1$ encircling only the poles \eqref{N2Pls} with $j=1$ counterclockwise. For example, with $\de\equiv |x_1-x_2|>0$, we can take the contour consisting of line segments $x_1+u\pm ia/2$, $|u|\leq\de/2$, and $x_1\pm\de/2+iv$, $|v|\leq a/2$. Introducing the function
\be
\label{varphi2}
\varphi_2(x,y)\equiv e^{\frac{2\pi i}{a} y_2(x_1+x_2)}\int_{\gamma_1}dz\frac{e^{\frac{2\pi i}{a}(y_1-y_2)z}}{\prod_{j=1}^2\prod_{n=0}^{m-1} 2\sinh\big(\frac{\pi}{a}(x_j-z+i(m-2n-1)/2)\big)},
\ee
the following result is an easy consequence of \eqref{Phi2Expr} and  Cauchy's residue theorem.

\begin{proposition}
\label{Prop:Phi2Exp}
For $m\in\mathbb{Z}_+$ and $a>m-1$, we have
\be
\Phi_2(x,y) = (-1)^m\frac{\prod_{n=1}^{m-1}2\sin(\pi n/a)}{a}\frac{\sum_{\sigma\in S_2}\varphi_2(\sigma(x),y)}{2\sinh\big(\frac{\pi}{a}(y_1-y_2)\big)}.
\ee
\end{proposition}

We proceed to establish the precise connection between $\varphi_2$ and the self-dual BA-function $\psi_2$. As detailed in the following proposition, when applying Cauchy's residue thm.~to \eqref{varphi2} we obtain a series expansion for $\varphi_2$ analogous to the expansion of $\psi_2$ given by the $N=2$ instance of \eqref{psiN}.

\begin{proposition}
\label{Prop:varphi2Exp}
Assuming that $m\in\mathbb{Z}_+$ and $a>m-1$, we have
\be
\varphi_2(x,y) =e^{\frac{2\pi i}{a}(x,y)}\sum_{l=0}^{m-1}\varphi_{2,l}(x)e^{\frac{2\pi}{a}\left(\frac{m-1}{2}-l\right)(y_1-y_2)}
\ee
with
\be
\label{varphi2l}
\varphi_{2,l}(x) = \frac{a(-i)^m}{\prod_{n\neq l}2\sin\big(\frac{\pi}{a}(n-l)\big)}\frac{1}{\prod_{n=0}^{m-1}2\sinh\big(\frac{\pi}{a}(x_2-x_1+i(n-l))\big)}.
\ee
\end{proposition}

We note that the above series essentially amounts to a special (terminating) case of the basic hypergeometric series
\be
{}_2\phi_1
\left(\,\begin{matrix}a,\, b\\ \, c\end{matrix}\,\Big|\,q,z\right) = \sum_{l=0}^\infty \frac{(a:q)_l(b:q)_l}{(c;q)_l(q;q)_l}z^l,
\ee
where the $q$-shifted factorial is given by $(a;q)_0=1$ and $(a;q)_l=(1-a)\cdots(1-aq^{l-1})$ when $l>0$. To verify this claims, it suffices to observe that
\be
\begin{split}
\frac{1}{\prod_{n\neq l}2\sin\left(\frac{\pi}{a}(n-l)\right)} &= \frac{1}{\prod_{n=1}^{m-1}2\sin\left(\frac{\pi}{a}n\right)} \prod_{n=0}^{l-1}\frac{2\sin\left(\frac{\pi}{a}(m-1-n)\right)}{2\sin\left(\frac{\pi}{a}(-n-1)\right)}\\
&= \frac{t_m^l}{\prod_{n=1}^{m-1}2\sin\left(\frac{\pi}{a}n\right)}\frac{(q^2/t_m^2;q^2)_l}{(q^2;q^2)_l}
\end{split}
\ee
and similarly that
\begin{multline}
\frac{1}{\prod_{n=0}^{m-1}2\sinh\big(\frac{\pi}{a}(x_2-x_1+i(n-l))\big)}\\
= t_m^l\frac{1}{\prod_{n=0}^{m-1}2\sinh\big(\frac{\pi}{a}(x_2-x_1+in)\big)}\frac{(q^2X_2/t_m^2X_1;q^2)_l}{(q^2X_2/X_1;q^2)_l},
\end{multline}
where we have used the parametrisation \eqref{q} for $q$ and introduced
\be
t_m\equiv e^{-\frac{i\pi}{a}m} = q^m
\ee
as well as
\be
X_j\equiv e^{\frac{2\pi}{a}x_j} = q^{i2x_j} ,\ \ Y_j\equiv e^{\frac{2\pi}{a}y_j} = q^{i2y_j},\ \ \ j = 1,2.
\ee
Indeed, it is clear from the above computations that
\be
\label{varphi2Exp}
\begin{split}
\varphi_2(x,y) &= \frac{a(-i)^m\exp\left(\frac{2\pi i}{a}(x-i\rho_2(m-1),y)\right)}{\prod_{n=1}^{m-1}2\sin(\pi n/a)\prod_{n=0}^{m-1}2\sinh\big(\frac{\pi}{a}(x_2-x_1+in)\big)}\\
& \quad \cdot {}_2\phi_1
\left(\,\begin{matrix}q^2/t_m^2,\, q^2X_2/t_m^2X_1\\ \, q^2X_2/X_1\end{matrix}\,\Big|\,q^2,t_m^2Y_2/Y_1\right).
\end{split}
\ee

Starting from the $q$-difference eigenvalue equation satisfied by ${}_2\phi_1$, it would now be a simple exercise to show that $\varphi_2$ is an eigenfunction of the $N=2$ Macdonald operator $D_2^1$ for suitable parameter values. One could also invoke the duality relation \eqref{dualRel} and linear independence of $\varphi_2(\sigma(x),y)$, $\sigma\in S_2$, for generic $x\in\mathbb{R}^2$; see Lemma \ref{Lemma:inv} below. Either way, one could then use a uniqueness result such as Prop.~4.1 in \cite{Nou23} for `asymptotically free' eigenfunctions of $D_2^1$ to arrive at the connection between $\varphi_2$ and $\psi_2$. We note that such uniqueness results for arbitrary $N$ can be found in the paper of Letzter and Stokman \cite{LS08} (in a general root system setting) as well as that of Noumi and Shiraishi \cite{NS12}.

Here, we rely instead on the defining properties of $\psi_2$ and show that $\varphi_2$ has vanishing properties similar to those given by \eqref{psiNqinv} with $N=2$. 

\begin{proposition}
For $s=1,\ldots,m-1$, we have
\be
\label{varphi2Diff}
\varphi_2\left(x,y+is(e_1-e_2)/2\right)-\varphi_2\left(x,y-is(e_1-e_2)/2\right) = 0\ \ \text{when}\ \ e^{\frac{2\pi}{a}(y_1-y_2)} = 1.
\ee
\end{proposition}

\begin{proof}
Up to an overall factor $2\exp(\frac{2\pi i}{a}(y_2(x_1+x_2)))$, the LHS of \eqref{varphi2Diff} is given by the integral
\be
\int_{\gamma_1}dz\frac{\exp\big(\frac{2\pi i}{a}(y_1-y_2)z\big)\sinh(\frac{\pi s}{a}(x_1+x_2-2z))}{\prod_{j=1}^2\prod_{n=0}^{m-1} 2\sinh(\frac{\pi}{a}(x_j-z+i(m-2n-1)/2))},
\ee
which, by Cauchy's theorem, equals
\begin{multline}
(-i)^me^{\frac{2\pi i}{a}(y_1-y_2)x_1}\sum_{l=0}^{m-1}\frac{a}{\prod_{n\neq l}2\sin(\frac{\pi}{a}(n-l))}\\
\cdot \frac{\sinh(\frac{\pi s}{a}(x_2-x_1-i(m-2l-1))}{\prod_{n=0}^{m-1}2\sinh(\frac{\pi}{a}(x_2-x_1+i(n-l))}e^{\frac{\pi}{a}(m-2l-1)(y_1-y_2)}.
\end{multline}
Since $e^{-\frac{\pi}{a}(m-2l-1)(y_1-y_2)}=e^{\frac{\pi}{a}(m-1)(y_1-y_2)}$ when $e^{\frac{2\pi}{a}(y_1-y_2)} = 1$, we only have to show that
\be
\label{sum}
\sum_{l=0}^{m-1}\frac{a}{\prod_{n\neq l}2\sin(\frac{\pi}{a}(n-l))} \frac{\sinh(\frac{\pi s}{a}(x_2-x_1+i(m-2l-1))}{\prod_{n=0}^{m-1}2\sinh(\frac{\pi}{a}(x_2-x_1+i(n-l))} = 0.
\ee
To this end, we note that $s<m$ entails that the left-hand side is a meromorphic function in $x_1-x_2$ that is bounded away from its poles and decays exponentially as $|\re(x_1-x_2)|\to\infty$. Hence, Liouville's theorem will imply the above identity once we show that the residues at the (simple) poles
\be
x_2-x_1+i\ell = i\ell^\prime a,\ \ \ \ell = -m+1,\ldots,m-1,\ \ \ell^\prime\in\mathbb{Z},
\ee
all vanish; and, by $ia$-(anti)periodicity, we may and shall restrict attention to $\ell^\prime=0$. For a fixed $\ell=-m+1,\ldots,m-1$, we observe that there exists $n=0,\ldots,m-1$ such that $n-l=\ell$ if and only if
\be
\label{nrange}
l = \max(-\ell,0),\ldots,\min(m-1,m-1-\ell).
\ee
It follows that the residue of the left-hand side of \eqref{sum} at $x_2-x_1+i\ell=0$ is proportional to
\be
\sum_{l=\max(-\ell,0)}^{\min(m-1,m-1-\ell)}f_l,\ \ \ f_l(a)\equiv \frac{2a}{\prod_{n\neq l}\sin(\frac{\pi}{a}(n-l))}\frac{\sin(\frac{\pi s}{a}(2l+\ell-m+1))}{\prod_{n\neq l+\ell}\sin(\frac{\pi}{a}(n-l-\ell))}.
\ee
On the set of integers \eqref{nrange}, we have the involution $\sigma: l\mapsto m-1-\ell-l$. Since $f_{\sigma(l)}=-f_l$, as is easily checked, this clearly implies that the residue vanishes.
\end{proof}

Setting $l=0$ in \eqref{varphi2l}, we find that
\be
\label{varphi20}
\varphi_{2,0}(x) = \frac{2a(-i)^m}{\prod_{n=1}^{m-1}2\sin(\pi n/a)}\frac{1}{\prod_{n=0}^{m-1}2\sinh\big(\frac{\pi}{a}(x_2-x_1+in)\big)}.
\ee
Comparing \eqref{varphi20}, \eqref{varphi2Diff} and \eqref{varphi2Exp} with the $N=2$ instances of \eqref{psiNrhom}, \eqref{psiNqinv} and \eqref{psiN}, respectively, the uniqueness of the BA-function, once a normalisation has been fixed, implies that
\be
\varphi_2(x,y) = \frac{2a(-i)^m}{\prod_{n=1}^{m-1}2\sin(\pi n/a)}\frac{\psi_2(\exp(-i\pi/a),m-1;i\sigma(x),iy)}{\prod_{n=-m+1}^{m-1}2\sinh\big(\frac{\pi}{a}(x_2-x_1+in)\big)}.
\ee
When combined with Prop.~\ref{Prop:Phi2Exp}, we obtain the $N=2$ case of Thm.~\ref{Thm:PhiN}; and a comparison with \eqref{varphi2} yields the $N=2$ instance of Prop.~\ref{Prop:psiN}.

\subsection{The inductive step}
We can now use the $N=2$ case as the basis for proofs by induction. More precisely, we consider $N\geq 3$ and assume that Thm.~\ref{Thm:PhiN} along with Props.~\ref{Prop:PhiNExp}--\ref{Prop:vphiExp}, Lemma \ref{Lemma:inv} and Props.~\ref{Prop:qinv}--\ref{Prop:varphiNpsi} hold true when $N$ is replaced by $N-1$. (In the $N=2$ case, the pertinent results have all been established above.)

To begin with, we require that $x=(x_1,\ldots,x_N)\in\R^N$ with $x_j\neq x_k$ for all $1\leq j<k\leq N$. Letting $\de = \min\{|x_j-x_k|\mid 1\leq j<k\leq N\}$, we construct $N$ contours $\gamma_j$, $j=1,\ldots,N$, from the line segments $x_j+u\pm ia/2$, $|u|\leq\de/2$, and $x_j\pm\de/2+iv$, $|v|\leq a/2$; and define the function
\begin{multline}
\label{varphiN}
\varphi_N(x,y)\equiv e^{\frac{2\pi i}{a} y_N\sum_{j=1}^Nx_j}\\
\cdot \int_{\underline{\gamma}}dz \prod_{1\leq j<k\leq N-1}\sinh\left(\frac{\pi}{a}(z_k-z_j)\right)\prod_{n=-m+1}^{m-1} 2 \sinh\left(\frac{\pi}{a}(z_k-z_j+in)\right)\\\cdot \frac{\varphi_{N-1}(z,(y_1-y_N,\ldots,y_{N-1}-y_N))}{\prod_{j=1}^N\prod_{k=1}^{N-1}\prod_{n=0}^{m-1}2\sinh\big(\frac{\pi}{a}\big(x_j-z_k+\frac{i}{2}(m-2n-1)\big)\big)},
\end{multline}
with $\underline{\gamma}\equiv \gamma_1\times\cdots\times\gamma_{N-1}$, and where the $z_j$-contour $\gamma_j$ only encircles the simple poles
\be
z_j = x_j+i(m-2n-1)/2,\ \ \ n = 0,\ldots,m-1.
\ee
The arbitrary-$N$ generalisation of Prop.~\ref{Prop:Phi2Exp} now follows.

\begin{proposition}
\label{Prop:PhiNExp}
Assuming $m\in\mathbb{Z}_+$ and $a>m-1$, we get
\begin{multline}
\label{PhiExp}
\prod_{1\leq j<k\leq N}2\sinh(\pi(y_j-y_k))\cdot \Phi_N(x,y)\\
= (-1)^{mN(N-1)/2}\left(\frac{\prod_{n=1}^{m-1}2\sin(\pi n/a)}{a}\right)^{N-1}\sum_{\sigma\in S_N}\varphi_N(\sigma(x),y).
\end{multline}
\end{proposition}

\begin{proof}
To ease the notation, we suppress dependence on the parameters $a$ and $m$ throughout the proof. Substituting  \eqref{PhiNRec} in the left-hand side of \eqref{PhiExp}, using \eqref{cW}--\eqref{cK} to reduce the size of the resulting expression and invoking the above statement after taking $N\to N-1$, we obtain
\begin{multline*}
\left(\frac{\prod_{n=1}^{m-1}2\sin(\pi n/a)}{a}\right)^{N-1}\prod_{j=1}^{N-1}2\sinh(\pi(y_j-y_N)) \frac{e^{\frac{2\pi i}{a} y_N\sum_{j=1}^Nx_j}}{(N-1)!}\\
\cdot\sum_{\sigma\in S_{N-1}}\int_{\mathbb{R}^{N-1}}dzW_{N-1}(z)\cS^\sharp_N(x,z)\varphi_{N-1}(\sigma(z),(y_1-y_N,\ldots,y_{N-1}-y_N)).
\end{multline*}
Thanks to the $S_{N-1}$-invariance of the domain of integration $\mathbb{R}^{N-1}$ as well as $W_{N-1}(z)$ and $\cS^\sharp_N(x,z)$, we can replace the sum over $\sigma\in S_{N-1}$ by a factor $(N-1)!$. By also substituting $2\sinh(\pi(y_j-y_N))=e^{\pi(y_j-y_N)}-e^{-\pi(y_j-y_N)}$, we thus obtain
\begin{multline}
\left(\frac{\prod_{n=1}^{m-1}2\sin(\pi n/a)}{a}\right)^{N-1}e^{\frac{2\pi i}{a} y_N\sum_{j=1}^Nx_j}\sum_{\de\in\{\pm 1\}^{N-1}}\de_1\cdots\de_{N-1}\\
\cdot \int_{\mathbb{R}^{N-1}}dzW_{N-1}(z)\cS^\sharp_N(x,z)e^{\pi\sum_{j=1}^{N-1}\de_j(y_j-y_N)}\varphi_{N-1}(z,(y_1-y_N,\ldots,y_{N-1}-y_N)).
\end{multline}
To rewrite the integral further we shall make use of the invariance property
\be
\label{varphiInvProp}
e^{\pi\sum_{j=1}^{N-1}\de_ju_j}\varphi_{N-1}(z,u) = \varphi_{N-1}(z-(ia/2)\de,u)\prod_{j=1}^{N-1}\delta_j^{m(N-2)}.
\ee
When $N-1=2$ it is a simple consequence of Prop.~\ref{Prop:varphi2Exp} and, using this as the basis for an induction argument, the validity of the formula in the general-$N$ case is readily inferred from \eqref{vphiExp2}--\eqref{Coeffs} and the observation that $W_{N-1}(z+(ia/2)\de)=W_{N-1}(z)$. Substituting \eqref{varphiInvProp} (with $u_j=y_j-y_N$) in the above integral, taking $z\to z+(ia/2)\de$, under which $W_{N-1}(z)$ is invariant, and noting that
\be
\cS^\sharp_N(x,z+(ia/2)\de) = \frac{(-1)^{mN(N-1)/2}\prod_{j=1}^{N-1}\delta_j^{mN}}{\prod_{j=1}^N\prod_{k=1}^{N-1}\prod_{n=0}^{m-1}2\sinh\big(\frac{\pi}{a}(x_j-z_k+i(m-2n-1)/2)\big)},
\ee
we arrive at
\begin{multline}
(-1)^{mN(N-1)/2}\left(\frac{\prod_{n=1}^{m-1}2\sin(\pi n/a)}{a}\right)^{N-1}e^{\frac{2\pi i}{a} y_N\sum_{j=1}^Nx_j}\sum_{\de\in\{\pm 1\}^{N-1}}\de_1\cdots\de_{N-1}\\
\cdot \int_{\R-\frac{ia}{2}\de_1}dz_1\cdots \int_{\R-\frac{ia}{2}\de_{N-1}}dz_{N-1}
\frac{W_{N-1}(z)\varphi_{N-1}(z,(y_1-y_N,\ldots,y_{N-1}-y_N))}{\prod_{j=1}^N\prod_{k=1}^{N-1}\prod_{n=0}^{m-1}2\sinh\big(\frac{\pi}{a}\big(x_j-z_k+\frac{i}{2}(m-2n-1)\big)\big)},
\end{multline}
where the sum of integrals amounts to
\be
\int_{\cC^{N-1}}dz\frac{W_{N-1}(z)\varphi_{N-1}(z,(y_1-y_N,\ldots,y_{N-1}-y_N))}{\prod_{j=1}^N\prod_{k=1}^{N-1}\prod_{n=0}^{m-1}2\sinh\big(\frac{\pi}{a}\big(x_j-z_k+\frac{i}{2}(m-2n-1)\big)\big)}
\ee
with $\cC$ the contour consisting of the lines $\R\pm ia/2$, traversed from right/left to left/right. By Cauchy's residue theorem, this one integral equals
\begin{multline*}
\sum_{\sigma\in S_N}\int_{\gamma_{\sigma(1)}}dz_1\cdots \int_{\gamma_{\sigma(N-1)}}dz_{N-1}\\
\cdot \frac{W_{N-1}(z)\varphi_{N-1}(z,(y_1-y_N,\ldots,y_{N-1}-y_N))}{\prod_{j=1}^N\prod_{k=1}^{N-1}\prod_{n=0}^{m-1}2\sinh\big(\frac{\pi}{a}\big(x_j-z_k+\frac{i}{2}(m-2n-1)\big)\big)} + R(x,y)
\end{multline*}
for some remainder term $R(x,y)$, resulting from residues of the integrand at points $z=(z_1,\ldots,z_{N-1})$ given by
\be
\label{zjSpec}
z_j = x_{i_j}+i(m-2n_j+1)/2,\ \ \ j = 1,\ldots,N-1,
\ee
where at least two $i_j$ coincide.

Referring back to \eqref{varphiN}, we find that it remains to prove that $R=0$. To this end, we take $N\to N-1$ in Prop.~\ref{Prop:varphiNpsi} and use the resulting formula to rewrite the integrand as
\be
C_{N-1}\frac{\psi_{N-1}(e^{-i\pi/a},m-1;iz,i(y_1-y_N,\ldots,y_{N-1}-y_N))\prod_{1\leq j<k\leq N-1}\sinh\big(\frac{\pi}{a}(z_j-z_k)\big)}{\prod_{j=1}^N\prod_{k=1}^{N-1}\prod_{n=0}^{m-1}2\sinh\big(\frac{\pi}{a}\big(x_j-z_k+\frac{i}{2}(m-2n-1)\big)\big)},
\ee
where $C_{N-1}$ is a constant (for fixed $m$ and $a$), whose value is of no consequence for the arguments below. From Prop.~4.4 in \cite{Cha02} (part (iii)), we know that $\psi_{N-1}(u,v)$ is an entire function. The presence of the factors $\sinh\big(\frac{\pi}{a}(z_j-z_k)\big)$ thus entails that (iterated) residues at points \eqref{zjSpec} such that $i_j=i_k$ and $n_j=n_k$ for some $1\leq j\neq k\leq N-1$ all vanish. In the remaining cases we have $n_j\neq n_k$ while $i_j=i_k$. Thanks to the $S_{N-1}$-invariance of the integrand (cf.~Lemma \ref{Lemma:inv}), we may and shall restrict attention to $j=1$ and $k=2$. Letting $\widetilde{\varphi}_{N-1}(z)=\psi_{N-1}(e^{-i\pi/a},m-1;iz,i(y_1-y_N,\ldots,y_{N-1}-y_N))$, $t=x_{i_1}=x_{i_2}$ and $s=n_1-n_2$, we observe that
\begin{multline*}
\tilde{\varphi}_{N-1}\big((t+i(m-2n_1-1)/2,t+i(m-2n_2-1)/2,\ldots)\big)\\
= \tilde{\varphi}_{N-1}\big((t+i(m-n_1-n_2-1)/2-is/2,t+i(m-n_1-n_2-1)/2+is/2,\ldots)\big).
\end{multline*}
Note that the self-duality of $\psi_{N-1}$ entails that $\widetilde{\varphi}_{N-1}(z)$ satisfies the vanishing conditions in Prop.~\ref{Prop:qinv} with $N\to N-1$ in $z$. Since each $n_j=0,\ldots,m-1$, we have $s=-m+1,\ldots m-1$. From the vanishing conditions for $\tilde{\varphi}$, we can thus infer that its specialisation at \eqref{zjSpec} with $i_1=i_2$ is invariant under the interchange $n_1\leftrightarrow n_2$. Moreover, we have
\be
\sinh(\pi(z_1-z_2)/a) = \sinh(i\pi(n_2-n_1)/a),
\ee
so that the same specialisation of $\varphi(z)W_{N-1}(z)$ is antisymmetric under $n_1\leftrightarrow n_2$. In fact, more generally, when $i_1=\cdots=i_M$, $M\geq 2$, coincide, we can, similarly, show that antisymmetry extends to any permutation of the corresponding $M$ distinct integers $n_1,\ldots,n_M$. (Essentially, we need only to decompose such a permutation in terms of (elementary) transpositions and appeal to the above discussion.) Either way, we see that the residues corresponding to a fixed choice of indices $i_j$, at least two of which coincide, and integers $n_j$, fixed up to permutations among the ones corresponding to equal $i_j$, cancel (in pairs) and consequently that $R=0$.
\end{proof}

We proceed to show that $\varphi_N$ \eqref{varphiN} essentially amounts to the BA-function reviewed above. This requires that we establish two things: the relevant series expansion and vanishing properties.

\begin{proposition}
\label{Prop:vphiExp}
We have
\be\label{vphiExp}
\varphi_N(x;y) = e^{\frac{2\pi i}{a}(x,y)}\sum_\nu \varphi_{N,\nu}(x)e^{\frac{2\pi}{a}(\nu,y)}
\ee
with the sum running over weight vectors of the form \eqref{nu} with $m\to m-1$ and where
\be\label{vphirho}
\begin{split}
\varphi_{N,\rho_N(m-1)}(x) &= \left(\frac{2a(-i)^m}{\prod_{n=1}^{m-1}2\sin(\pi n/a)}\right)^{N(N-1)/2}\\
&\quad \cdot \frac{1}{\prod_{1\leq j<k\leq N}\prod_{n=0}^{m-1}2\sinh\big(\frac{\pi}{a}(x_k-x_j+in)\big)}.
\end{split}
\ee
\end{proposition}

\begin{proof}
By our induction assumption, we may invoke the proposition after taking $N\to N-1$. Substituting the resulting expansion \eqref{vphiExp} in \eqref{varphiN} and making use of Cauchy's residue theorem, we deduce
\begin{multline}\label{vphiExp2}
\varphi_N(x;y) = e^{\frac{2\pi i}{a} y_N\sum_{j=1}^Nx_j} \sum_{\nu^\prime} e^{\frac{2\pi}{a}(\nu^\prime,y)}\\
\cdot \int_{\underline{\gamma}}dz\frac{W_{N-1}(z)\varphi_{N-1,\nu^\prime}(z)e^{\frac{2\pi i}{a}(z,(y_1-y_N,\ldots,y_{N-1}-y_N))}}{\prod_{j=1}^N\prod_{k=1}^{N-1}\prod_{n=0}^{m-1}2\sinh\big(\frac{\pi}{a}(x_j-z_k+i(m-2n-1)/2)\big)}\\
= e^{\frac{2\pi i}{a}(x,y)} \sum_{\nu^\prime} e^{\frac{2\pi}{a}(\nu^\prime,(y_1,\ldots,y_{N-1}))}
\sum_{n_1=0}^{m-1}\cdots \sum_{n_{N-1}=0}^{m-1}C_{\nu^\prime;n_1,\ldots,n_{N-1}}(x)e^{\frac{2\pi}{a}\sum_{k=1}^{N-1}(n_k-(m-1)/2)(y_k-y_N)},
\end{multline}
with $\nu^\prime$ running over the set of weight vectors obtained after taking $N\to N-1$ and $m\to m-1$ in \eqref{nu}, and where
\begin{multline}
\label{Coeffs}
C_{v^\prime;n_1,\ldots,n_{N-1}}(x)\\
= (-ia)^{N-1}\frac{W_{N-1}(x_1+i(m-2n_1-1)/2,\ldots,x_{N-1}+i(m-2n_{N-1}-1)/2)}{\prod_{k=1}^{N-1}\prod_{n^\prime\neq n_k}2\sinh\big(\frac{i\pi}{a}(n_k-n^\prime)\big)}\\
\cdot \frac{\varphi_{N-1,\nu^\prime}(x_1+i(m-2n_1-1)/2,\ldots,x_{N-1}+i(m-2n_{N-1}-1)/2)}{\prod_{k=1}^{N-1}\prod_{j\neq k}\prod_{n^\prime=0}^{m-1}2\sinh\big(\frac{\pi}{a}(x_j-x_k+i(n_k-n^\prime)\big)}
\end{multline}
and we have used that $(\nu^\prime,(y_N,\ldots,y_N))=0$, since $\nu^\prime$ is a linear combination of vectors $e_j-e_k$ with $1\leq j<k\leq N-1$. Setting $l_{j,N-1}=m-1-n_j=0,\ldots,m-1$, we get the equality
\begin{multline}
\sum_{1\leq j<k\leq N-1}\left(\frac{m-1}{2}-l_{jk}\right)(e_j-e_k) + \sum_{j=1}^{N-1}\left(n_j-\frac{m-1}{2}\right)(e_j-e_N)\\
= \sum_{1\leq j<k\leq N}\left(\frac{m-1}{2}-l_{jk}\right)(e_j-e_k)
\end{multline}
and the validity of \eqref{vphiExp} with $\nu$ from \eqref{nu} clearly follows.

Since the only term in \eqref{vphiExp2} that contributes to $\varphi_{N,\rho_{N}(m-1)}(x)$ corresponds to $\nu^\prime=\rho_{N-1}(m-1)$ and $n_1=\cdots=n_{N-1}=m-1$, we have
\begin{multline}
\varphi_{N,\rho_N(m-1)}(x) = \left(\frac{a(-i)^m}{\prod_{n=1}^{m-1}2\sin(\pi n/a\big)}\right)^{N-1}\\
\cdot \frac{W_{N-1}(x)\varphi_{N-1,\rho_{N-1}(m-1)}(x)}{\prod_{k=1}^{N-1}\prod_{j\neq k}\prod_{n^\prime=0}^{m-1}2\sinh\big(\frac{\pi}{a}(x_j-x_k+i(m-1-n^\prime)\big)},
\end{multline}
where we have used the fact that both $W_{N-1}(x)$ and $\varphi_{N-1,\rho_{N-1}(m-1)}$ are invariant under translations $x\to x+(t,\ldots,t)$ for any $t\in\C$. Using \eqref{cW} and \eqref{vphirho} for $N\to N-1$, we easily arrive at the right-hand side of \eqref{vphirho}.
\end{proof}

To establish quasi-invariance, we make use of the fact that $\varphi_N$ is $S_N$-invariant in the sense of the following lemma.

\begin{lemma}
\label{Lemma:inv}
We have
\be
\varphi_N(\sigma(x),\sigma(y)) = (-)^{|\sigma|}\varphi_N(x,y),\ \ \ \sigma\in S_N.
\ee
\end{lemma}

\begin{proof}
From \eqref{varphiN} and the induction assumption, we see immediately that the statement holds true whenever $\sigma\in S_{N-1}$. Hence, we need only to consider the special case $\sigma=\sigma_{N-1}$. By a direct computation, similar to the one in the proof of Proposition \ref{Prop:vphiExp}, we obtain
\be
\begin{split}
\varphi_{N,\sigma_{N-1}(\rho_N(m-1))}(x) &= - \left(\frac{a(-i)^m}{\prod_{n=1}^{m-1}2\sin(\pi n/a)}\right)^{N(N-1)/2}\\
&\quad \cdot \frac{1}{\prod_{k=1}^N\prod_{j=1}^{\min(k,N-1)-1}\prod_{n=0}^{m-1}2\sinh\big(\frac{\pi}{a}(x_k-x_j+in\big)}\\
&\quad \cdot \frac{1}{\prod_{n=0}^{m-1}2\sinh\big(\frac{\pi}{a}(x_{N-1}-x_N+in\big)}\\
&= -\varphi_{N,\rho_N(m-1)}(\sigma_{N-1}(x)).
\end{split}
\ee
(Note that only the term in \eqref{vphiExp2} corresponding to $\nu^\prime=\rho_{N-1}$, $n_1=\cdots=n_{N-2}=m-1$ and $n_{N-1}=0$ contributes.) It follows that $\varphi_N(x,y)$ and $-\varphi_N(\sigma_{N-1}(x),\sigma_{N-1}(y))$ are both (finite) series of the form
\be
e^{\frac{2\pi i}{a}(x-i\rho_N(m-1),y)}f_N(x,y),\ \ \ f_N(x,y) = \Gamma_0(x)+\sum_{\nu>0}\Gamma_\nu(x)e^{-\frac{2\pi}{a}(\nu,y)},
\ee
where $\Gamma_0=\varphi_{N,\rho_N(m-1)}$ and $\nu>0$ means that $\nu=\sum_{1\leq j<k\leq N}l_{jk}(e_j-e_k)$ with $l_{jk}\geq 0$ and not all equal to zero. From Proposition \ref{Prop:vphiExp}, we get that the functions $\varphi_N(\sigma(x),y)$, $\sigma\in S_N$, form a linearly independent set; and,
since the pre-factor $\prod_{j<k}2\sinh(\pi(y_j-y_k))$ in \eqref{PhiExp} is (anti)periodic under shifts $y_j\to y_j-i$ and the Macdonald operators $D_N^r$ are $S_N$-invariant, it follows that $\varphi_N(x,y)$ and $-\varphi_N(\sigma_{N-1}(x),\sigma_{N-1}(y))$ satisfy a difference equation
\be
D_N(y)\varphi(x,y) = \varphi(x,y)\cdot \text{const.}\sum_{j=1}^N e^{\frac{2\pi}{a}x_j},\ \ \ D_N(y)\equiv D^1_N(q^{i2y};q^2,q^{2(1-m)}).
\ee
For the corresponding series $f_N$, this translates into a difference equation $L_N(y)f(x,y)=f(x,y)\cdot \text{const.}\sum_{j=1}^Ne^{\frac{2\pi}{a}x_j}$ given by the difference operator
\be
\begin{split}
L_N(y) &\equiv e^{-\frac{2\pi i}{a}(x-i\rho_N(m-1),y)} D_N(y) e^{\frac{2\pi i}{a}(x-i\rho_N(m-1),y)}\\
&= \sum_{j=1}^Ne^{\frac{2\pi}{a}x_j}\prod_{k<j}\frac{1-e^{\frac{2\pi}{a}(y_j-y_k+i(m-1))}}{1-e^{\frac{2\pi}{a}(y_j-y_k)}}\cdot \prod_{k>j}\frac{1-e^{\frac{2\pi}{a}(y_k-y_j-i(m-1))}}{1-e^{\frac{2\pi}{a}(y_k-y_j)}}\cdot e^{-i\partial_{y_j}}.
\end{split}
\ee
By expanding the coefficients in power series in $e^{\frac{2\pi}{a}(y_k-y_j)}$, $1\leq j<k\leq N$, this, together with the leading coefficient $\Gamma_0$, is readily seen to uniquely determine $f_N$. This essentially amounts to the uniqueness result in Theorem 2.3 in \cite{NS12}, wherein a detailed proof can be found.
\end{proof}

\begin{proposition}
\label{Prop:qinv}
For all $1\leq j<k\leq N$ and $s=1,\ldots,m-1$, we have the vanishing property
\be
\varphi_N\left(x;y+\frac{is}{2}(e_j-e_k)\right) = \varphi_N\left(x;y-\frac{is}{2}(e_j-e_k)\right),\ \ \ e^{\frac{2\pi}{a}(y_j-y_k)} = 1.
\ee
\end{proposition}

\begin{proof}
Thanks to Lemma \ref{Lemma:inv}, it suffices to establish the vanishing property when $j=1$ and $k=2$; and, in this case, it follows immediately from \eqref{varphiN} and our induction assumption.
\end{proof}

Comparing Props.~\ref{Prop:vphiExp} and \ref{Prop:qinv} with the properties \eqref{psiN} and \eqref{psiNqinv}--\eqref{psiNrhom}, which uniquely characterises the self-dual BA-function $\psi_N$, we find the following result.

\begin{proposition}
\label{Prop:varphiNpsi}
Given $m\in\mathbb{Z}_+$ and $a>m-1$, we have
\be
\begin{split}
\varphi_N(a,m;x,y) &= \left(\frac{a(-i)^m}{\prod_{n=1}^{m-1}2\sin(\pi n/a)}\right)^{N(N-1)/2}\\
&\quad \cdot \frac{\psi_N(\exp(-i\pi/a),m-1;ix,iy)}{\prod_{1\leq j<k\leq N}\prod_{n=-m+1}^{m-1}2\sinh\big(\frac{\pi}{a}(x_k-x_j+in)\big)}.
\end{split}
\ee
\end{proposition}

Finally, by combining Props.~\ref{Prop:PhiNExp} and \ref{Prop:varphiNpsi}, we arrive at the statement in Thm.~\ref{Thm:PhiN} for arbitrary $N$; and, by using the above expression to express $\varphi_N$ and $\varphi_{N-1}$ in \eqref{varphiN} in terms of $\psi_N$ and $\psi_{N-1}$, respectively, we easily verify the claim in Prop.~\ref{Prop:psiN}.

\section*{Acknowledgements}
I would like to thank N.~Belousov for helpful comments. Financial support from the Swedish Research Council (Project-id 2024-05649\_VR) is gratefully acknowledged.

\bibliographystyle{amsalpha}

\end{document}